\documentclass{amsart}
\usepackage{graphicx}
\usepackage{amsfonts}
\usepackage{float} 
\newtheorem{tm}{Theorem}

\newtheorem{defi}{Definition}

\newtheorem{rem}{Remark}
\newtheorem{rems}{Remarks}
\newtheorem{lm}{Lemma}

\newtheorem{nota}{Notation}

\newtheorem{p}{Problem}

\begin{document}

\title{The disconnectedness of certain sets defined after
  uni-variate polynomials}
\author{Vladimir Petrov Kostov}
\address{} 
\email{}

\begin{abstract}
  We consider the set of monic real uni-variate polynomials of a given degree
  $d$
  with non-vanishing coefficients, with given signs of the coefficients and
  with given quantities $pos$ of their positive and $neg$ of their negative
  roots (all roots are distinct). For $d\geq 6$ and for signs of the
  coefficients
  $(+,-,+,+,\ldots ,+,+,-,+)$, we prove that the set of such polynomials having
  two positive, $d-4$ negative and two complex conjugate roots,
  is not connected. For $pos+neg\leq 3$ and for any $d$,
  we give the exhaustive answer to the
  question for which signs of the coefficients there exist polynomials with
  such values of $pos$ and $neg$.
  
   {\bf Key words:} real polynomial in one variable; hyperbolic polynomial;
  Descartes' rule of signs; discriminant set\\ 

  {\bf AMS classification:} 26C10; 30C15\\

  Author's address: Universit\'e C\^ote d’Azur, CNRS, LJAD, France\\
  
  email: vladimir.kostov@unice.fr
\end{abstract}
\maketitle

\section{Introduction}

We consider questions about the general family of monic uni-variate real
degree $d$ polynomials: $Q_d:=x^d+\sum _{j=0}^{d-1}a_jx^j$. In the space
$\mathbb{R}^d$ of the coefficients $a_j$ one defines the {\em discriminant set}
$\Delta _d$ as the set of their values for which the
polynomial $Q_d$ has a multiple real root. More precisely, if $\Delta ^1_d$
is the set of values of the coefficients for which $Q_d$ has a multiple root
(real or complex), then this is the set of the zeros of the determinant of the 
Sylvester matrix of the polynomials $Q_d$ and $Q_d'$. One has to set
$\Delta _d:=\Delta _d^1\setminus \Delta _d^2$, where $\Delta _d^2$ is the
set of values of the coefficients $a_j$ for which there is a multiple
complex conjugate pair of roots of $Q_d$ and no multiple real root.
It is true that dim$(\Delta _d)=$dim$(\Delta _d^1)=d-1$ and
dim$(\Delta ^2_d)=d-2$.

The set

$$R_{1,d}:=\mathbb{R}^d\setminus \Delta _d$$
consists of
$[d/2]+1$ open components
of dimension~$d$ ($[.]$ stands for the integer part of).
The polynomials $Q_d$ from a given component have one and the same
number~$\mu$ of
real roots (which are all distinct); the number~$\nu$ of complex conjugate pairs
can range from $0$ to $[d/2]$, because $\mu +2\nu =d$.
Given two polynomials with one and the
same number~$\nu$, one can continuously deform the roots of
the first polynomial into the roots of the second one by keeping the real roots
distinct throughout the deformation. This proves that to any possible
number~$\nu$ corresponds exactly one component of the set~$R_{1,d}$.

In the same way one can consider the components of the set

$$R_{2,d}:=\mathbb{R}^d\setminus (\Delta _d\cup \{ a_0=0\} )~.$$
The polynomials from one and the same open component
(also of dimension~$d$) have one and the same numbers $pos$ of
positive and $neg$ of negative roots (and no vanishing roots). When deforming 
the roots of one polynomial into the roots of another one,
one has to keep the same numbers $pos$ and $neg$ 
throughout the deformation. To each pair ($pos$, $neg$)
corresponds exactly one component of the set~$R_{2,d}$. As
$pos+neg=\mu$, $0\leq pos$, $neg\leq \mu$ and $\mu +2\nu =d$, there are

$$(d+1)+(d-1)+(d-3)+\cdots =([d/2]+1)([(d+1)/2]+1)$$
components of the set $R_{2,d}$.

A more complicated task is to study the components of the set

$$R_{3,d}:=\mathbb{R}^d\setminus (\Delta _d\cup
\{ a_0=0\} \cup \{ a_1=0\} \cup \cdots \cup \{ a_{d-1}=0\} )$$
of monic uni-variate polynomials with no multiple real roots and no zero
coefficients.

\begin{defi}\label{defisignpattern}
  {\rm A {\em sign pattern} of length $d+1$ is a sequence of $d+1$
    symbols $+$ and/or $-$ beginning with a~$+$. We say that a polynomial
    $Q_d$ with no vanishing coefficients defines the sign pattern
    $\sigma _0:=(+,\beta _{d-1},\beta _{d-2},\ldots ,\beta _0)$,
    $\beta _j=+$ or~$-$, 
    (notation: $\sigma (Q_d)=\sigma _0$), if sign$(a_j)=\beta _j$,
    $j=0$, $\ldots$, $d-1$.}
  \end{defi}
  
One can ask the question to which couple (sign pattern, pair $(pos, neg)$) 
%${\rm (signs~of~}a_j,~pos,~neg)$ 
(we call them {\em couples} for
short) corresponds at least one component of the set $R_{3,d}$.
%all coefficients are non-vanishing. 
The polynomials from a given component of $R_{3,d}$ have one and the same
couple. All components are of dimension~$d$.

When considering the set 
$R_{3,d}$, it is self-understood that the couples 
have to be defined in accordance with Descartes' rule of signs. 
This rule states that a real uni-variate polynomial $Q_d$
has not more positive roots
counted with multiplicity than the number $c$ of sign changes in the sequence of
its coefficients; the difference $c-pos$ is even,
see~\cite{Ca}, \cite{Cu}, \cite{DG}, \cite{Des}, \cite{Fo},
\cite{Ga}, \cite{J}, \cite{VJ}, \cite{La} or~\cite{Mes}.
Hence the sign of the constant term is
$(-1)^{pos}$. When the polynomial has no zero coefficients,
Descartes' rule of signs applied to $Q_d(-x)$ implies that $Q_d$
has not more negative roots counted with multiplicity than the number $p$
of sign preservations in that sequence (hence $c+p=d+1$), and the difference
$p-neg$ is also even.

\begin{defi}\label{deficompatible}
  {\rm A pair $(pos, neg)$ satisfying these conditions w.r.t.
a given sign pattern $\sigma _0$ is called {\em compatible} with~$\sigma _0$
(and vice versa), and the couple $(\sigma _0, (pos, neg))$
is also called {\em compatible}.
For a monic polynomial $Q_d$ with no vanishing coefficients, with $pos$
positive simple and $neg$ negative simple
roots and no other real roots, we say that $Q_d$ {\em realizes} the couple
$(\sigma (Q_d), (pos, neg))$.}
  \end{defi}

Yet this compatibility is just a 
necessary condition which turns out not to be sufficient. That is, there exist
cases when to certain compatible couples correspond no components of $R_{3,d}$.
So we formulate the first problem which we consider in the present paper:

\begin{p}\label{Problem1}
  For a given degree $d$, for which compatible couples do there exist
  monic polynomials realizing these couples? In other words, to which
  of the compatible couples there corresponds at least one component
  of the set $R_{3,d}$?
  \end{p}

Some results in relationship with Problem~\ref{Problem1}
are formulated in the next section. The problem seems to have been 
stated for the first time in~\cite{AJS}. The first
example when to a compatible couple corresponds no component
of the set $R_{3,d}$ (this is an example with $d=4$),
and the exhaustive answer to the problem for $d=4$,
are to be found in~\cite{Gr}. For $d=5$ and $d=6$, the result is given
in~\cite{AlFu}. For $d=7$ and partially for $d=8$ (resp. completely for
$d=8$), the answer is
formulated and proved in~\cite{FoKoSh} and~\cite{FoKoSh1}
(resp. in~\cite{KoCzMJ}). Different aspects concerning Descartes' rule of signs
are treated in papers \cite{KoMB},
\cite{CGK}, \cite{CGK1}, \cite{CGK2}, \cite{CGK3} and~\cite{FoNoSh}.

Of particular importance is the class of 
{\em hyperbolic polynomials}, i.~e. 
real polynomials whose roots are all real. The {\em hyperbolicity domain}
$\Pi _d$ is the set of values of the coefficients $a_j$ for which the
polynomial $Q_d$ is hyperbolic. For properties of hyperbolic polynomials
and the domain $\Pi _d$ see \cite{Ar}, \cite{Gi}, \cite{KoProcRSE}, \cite{Me}
and~\cite{Ko}.

In what follows we are also interested in another problem:

\begin{p}\label{Problem2}
  For a given degree $d$, to which compatible couples correspond two or 
  more components of the set $R_{3,d}$?
\end{p}

To formulate our first result connected with Problem~\ref{Problem2}
we introduce the following notation:

\begin{nota}\label{notaSP}
  {\rm For $d\geq 4$, we consider 
    $\mathbb{R}^d$ as the set
    $\{ (a_{d-1}, a_{d-2},\ldots ,a_0)|a_j\in \mathbb{R}\}$
    of $d$-tuples of coefficients (excluding the leading one)
    of polynomials 
    $Q_d$.
    %For a given polynomial $Q$ with non-zero
    %coefficients, we denote by $\sigma (Q)$ the {\em sign pattern}
    %defined by $Q$, i.~e.}
%$$\sigma (Q)~:=~(~+,~{\rm sign}(a_5),~{\rm sign}(a_4)~,\ldots ,~
     % {\rm sign}(a_0)~)$$
     % {\rm (the initial sign $+$ is the sign of the leading coefficient~$1$).
      %  More generally, a sign pattern is a sequence of $7$ signs $+$ or $-$
       % beginning with a~$+$.
        We denote by $\sigma _{\bullet}$
  the sign pattern $(+,-,+,+,\ldots ,+,+,-,+)$ and  
  by $\Pi ^*_d(\sigma _{\bullet})$ (resp. by $A(\sigma _{\bullet},(2,d-4))$)
  the subset of
  $\mathbb{R}^d$ of polynomials with signs of the coefficients (all
  non-zero) as defined by
  $\sigma _{\bullet}$ and having four positive and $d-4$ negative distinct
  real roots (resp. two positive and $d-4$ negative distinct roots and one
  complex conjugate pair). Hence the polynomials of the set
  $\Pi ^*_d(\sigma _{\bullet})$ are hyperbolic while the ones of the set
  $A(\sigma _{\bullet},(2,d-4))$ are~not.}
  \end{nota}

The following theorem
is proved in Section~\ref{secprtmbasic}.

\begin{tm}\label{tmbasic}
  (1) For $d\geq 6$, the set $A(\sigma _{\bullet},(2,d-4))$ is non-empty and
  consists of more than one component of the
  set $R_{3,d}$. Hence the set $A(\sigma _{\bullet},(2,d-4))$ is not connected.

  (2) For $d=4$ and $5$, the respective sets $A(\sigma _{\bullet},(2,0))$ and
  $A(\sigma _{\bullet},(2,1))$ are connected.
\end{tm}

\begin{rems}\label{remstmbasic}
  {\rm (1) One can mention cases in which the components of the set $R_{3,d}$ are
contractible and to each compatible couple corresponds exactly one 
component
of the set $R_{3,d}$ (see~\cite{Koconnect}). Namely, such are the cases of
hyperbolic polynomials and
of polynomials having exactly one or no
real roots at~all.

%(2) Theorem~\ref{tmbasic} implies that when there is at least
%one complex conjugate pair or when {remstmbasic}there are
%at least four real roots, one should expect to
%find cases when to some couple correspond at least two components
%of~$R_{3,d}$.

(2) In the case of polynomials having exactly two real distinct roots
(hence $pos+neg=2$)
to each compatible couple corresponds either one or no
component of~$R_{3,d}$, and all components are contractible. See more details
in the next section or in~\cite{Koconnect}. Whether in the case of exactly three real roots to each compatible couple corresponds at most one component of the set $R_{3,d}$ is an open question.

(3) For $d=4$ and $d=5$, pictures of the set $\Delta _d^1$ (from which one can deduce the form of the set 
$A(\sigma _{\bullet},(2,d-4))$) can be found in
\cite{KoSh} and \cite{CGK3} respectively.}

%(3) It is an open question whether for $pos+neg=3$, there are cases when to a
%given compatible couple corresponds more than one component of~$R_{3,d}$.}
  \end{rems}

\section{Comments and further results\protect\label{seccomments}}

%A particular subject concerning hyperbolic polynomials is the comparison
%between the moduli of their positive and their negative roots. That is,
%one can formulate the following problem:

%\begin{p}\label{Problem3}
%When these
%moduli (supposed all distinct) are ordered on the real positive
%half-line, at which
%positions can the moduli of the negative roots be, depending on the signs of
%the coefficients of the polynomial?
%\end{p}

%The problem has been treated in \cite{KoPuMaDe},
%\cite{KoSe}, \cite{KorigMO}, \cite{GKT} 
%and~\cite{Kocanon}.
%(As we shall see below, the question to compare the
%moduli of the positive and negative roots arises also for non-hyperbolic
%polynomials.)
%In the proof of Theorem~\ref{tmbasic} the basic result
%of~\cite{Kocanon} is used
%which we reproduce here.
%For arbitrary degree $d$, sign patterns are
%defined by complete analogy with Notation~\ref{notaSP}.

Given a sign pattern $\hat{\sigma}$ with $c$ sign changes and $p$
sign preservations
(hence $c+p=d$), Descartes' rule of signs implies that any
hyperbolic polynomial with sign pattern $\hat{\sigma}$ has exactly $c$ positive
and exactly $p$ negative roots counted with multiplicity. We define the 
{\em canonical order of moduli} corresponding to~$\hat{\sigma}$. 
The sign pattern $\hat{\sigma}$ is read from the right and
to each sign change (resp. sign preservation) one puts in correspondence
the letter $P$ (resp. the letter $N$).

For example, for
$\hat{\sigma}=\sigma _{\dagger}:=(+,-,-,-,+,+)$ (resp. for
$\hat{\sigma}=\sigma _{\bullet}$) this gives 
the string $NPNNP$ (resp. $PPNN\cdots NNPP$, $d-4$ times $N$).
After this one inserts
the symbol $<$ between any two consecutive letters which in the cases
of $\sigma _{\dagger}$ and $\sigma _{\bullet}$ gives

$$N<P<N<N<P~~~\, {\rm and}~~~\, 
P<P<N<N<\cdots <N<N<P<P$$
respectively. If one denotes by $\alpha _j$ and $\beta _j$
the moduli of the positive and negative roots, then one replaces the letters
$P$ and $N$ by these moduli which 
in the case of $\sigma _{\dagger}$ defines the canonical order

$$\beta _1<\alpha _1<\beta _2<\beta _3<\alpha _2$$
whereas the canonical order corresponding to $\sigma _{\bullet}$
is given by~(\ref{eqalphabeta}).

It is true that for any sign pattern $\sigma _0$ of length $d+1$,  
there exists a degree $d$ monic hyperbolic polynomial $T$ 
with $\sigma (T)=\sigma _0$ whose roots define the respective
canonical order of moduli, see Proposition~1 in~\cite{KoSe}.

%We formulate now the basic result of~\cite{Kocanon}. Suppose that 
%a given sign pattern contains no four consecutive signs

%$$(+,-,-,+)~,~~~(-,+,+,-)~,~~~(+,+,-,-)~~~{\rm or}~~~(-,-,+,+)$$
%(which is the case of $\sigma _{\bullet}$, but not of $\sigma _{\dagger}$).
%One can say in an equivalent way that the sign pattern
%contains no isolated sign change and no isolated sign preservation.

%\begin{tm}\label{tmcanonical}
%  If such a sign pattern is realizable by a hyperbolic
%  polynomial $T$, then
%the order of the moduli of the roots of $T$ 
%is the canonical one.
%\end{tm}
%In other words, for sign patterns without isolated
%sign changes and sign preservations, there exist no hyperbolic polynomials
%with order of the moduli of their roots other than the canonical one.

Our next step is to consider the cases when the polynomial $Q_d$ has not more
than three real roots, i.~e. $pos+neg\leq 3$ (and hence in the case of equality
the possible values of
the pair $(pos, neg)$ are $(3,0)$, $(2,1)$, $(1,2)$ and $(0,3)$).
For the cases $pos=neg=0$
and $pos+neg=1$, see part~(1) of Remarks~\ref{remstmbasic}. For $pos+neg=2$
(hence $d$ is even), we
remind some of the results of~\cite{Koconnect}.

\begin{defi} {\rm For $pos+neg=2$, we define {\em Case 1)} (resp.
  {\em Case 2)})
by the conditions the constant term to be positive,
all coefficients of monomials of odd degree to be positive (resp. negative),
the pair $(pos, neg)$ to equal $(2,0)$ (resp. $(0,2)$) and the coefficient
of at least one
monomial of even degree to be negative.}
  \end{defi}

\begin{tm}\label{tm2realroots}
  {\rm (see \cite{Koconnect}).} For $d$ even and $pos+neg=2$,

  (1) A given compatible couple is realizable if
  and only if it does not correspond to Case 1) or~2).

  (2) If the constant term is positive (hence $(pos, neg)=(2,0)$ or $(0,2)$)
  and one is not in Case 1) or~2), a given compatible couple
  is realizable by polynomials having any
  ratio different from $1$ between the moduli of the two real roots.

  (3) If the constant term is negative (hence $(pos, neg)=(1,1)$) and there
  are two monomials of odd degree with coefficients of opposite signs, then
  such a compatible couple is realizable by polynomials with any ratio
  of the moduli $\alpha$ and $\beta$ of its positive and negative root
  respectively.

  (4) If the constant term is negative and all coefficients of monomials
  of odd degree are positive (resp. negative), then such a compatible couple
  is realizable by polynomials with any ratio $\alpha /\beta <1$
  (resp. $\alpha /\beta >1$) and not realizable by polynomials with
  $\alpha /\beta \geq 1$ (resp. $\alpha /\beta \leq 1$).
  \end{tm}

To formulate the new results about the situation with $pos+neg=3$
we introduce the following notion:

\begin{defi}\label{defiZaction}
  {\rm For a given degree $d$, the
    $\mathbb{Z}_2\times \mathbb{Z}_2$-{\em action}
    on the set of compatible couples is defined by two commuting involutions.
    The first
    of them maps a polynomial $Q_d$ into $(-1)^dQ_d(-x)$ (this changes the pair
    $(pos, neg)$ into $(neg, pos)$, it changes the signs of the coefficients of
    $x^{d-1}$, $x^{d-3}$, $\ldots$ and preserves the signs of the other
    coefficients). The second involution maps $Q_d$ into $x^dQ_d(1/x)/Q_d(0)$
    (the pair $(pos, neg)$ is preserved and the sign pattern, eventually multiplied by $-1$, is read from
    the right; the roots of $x^dQ_d(1/x)/Q_d(0)$ are the reciprocals of
    the roots of $Q_d$).
An {\em orbit} of the
    $\mathbb{Z}_2\times \mathbb{Z}_2$-action consists of
    $2$ or $4$ compatible couples which are simultaneously realizable
    or not. This
    allows to formulate the results only for one of the $2$ or $4$ couples of
    a given orbit.}
  \end{defi}

    \begin{tm}\label{tm21}
 Suppose that the pair $(2,1)$ is compatible with the
 sign pattern $\sigma _{\triangle}$ (hence the constant term is positive).
 Then

      (1) The couple $\mathcal{C}:=(\sigma _{\triangle},(2,1))$ is realizable.

  Denote by $-\beta<0$, $\alpha _1>0$ and $\alpha _2>0$
  the three
  real roots of a polynomial realizing the couple
  $\mathcal{C}$.

  (2) If there are monomials
  $x^{2m}$ and $x^{2n-1}$ with negative coefficients (one can have
  $2m<2n-1$ or $2n-1<2m$), then for any of the five
  possibilities

  $$\beta <\alpha _1<\alpha _2~,~~~\, \, \beta =\alpha _1<\alpha _2~,~~~\, \,
  \alpha _1<\beta <\alpha _2~,~~~\, \, \alpha _1<\alpha _2=\beta ~~~\, \,
  {\rm and}~~~\, \, 
  \alpha _1<\alpha _2<\beta~,$$
  there exist polynomials realizing the couple $\mathcal{C}$.

  (3) If all odd monomials have positive coefficients, then only the
  possibility $\beta <\alpha _1<\alpha _2$ is realizable. 

(4) If all even monomials have positive coefficients, then only the
  possibility $\alpha _1<\alpha _2<\beta$ is realizable.
    \end{tm}

    The theorem is proved in Section~\ref{secprtm21}. The compatibility of
    the sign pattern with the pair $(2,1)$ implies that in part (3) (resp.
    in part (4)) of the theorem there is at least one even (resp. odd)
    monomial whose coefficient is negative.

    \begin{nota}\label{notaD}
  {\rm For $d$ odd, we denote by $D(a,b,c)$ the sign pattern consisting of
    $2a$ pluses followed by $b$ pairs ``$-,+$'' followed by $2c$ minuses,
    where $1\leq a$, $1\leq b$, $1\leq c$ and $2a+2b+2c=d+1$.}
\end{nota}

    \begin{tm}\label{tmD}
      Suppose that the pair $(3,0)$ is compatible with the
  sign pattern $\sigma _{\diamond}$ which is not of the form $D(a,b,c)$.
  Then the couple $(\sigma _{\diamond},(3,0))$ 
  is realizable.
    \end{tm}

    The theorem is proved in Section~\ref{secprtmD}.

    \begin{tm}\label{tmDbis}
      For $j=1$, $2$, $\ldots$, $b$, the couple
      $(D(a,b,c),(2j+1,0))$ is not realizable.
    \end{tm}

    The theorem is proved in Section~\ref{secprtmDbis}. Its proof resembles
    the proof of part~(i) of Theorem~4 in~\cite{KoSh} which treats a particular
    case of Theorem~\ref{tmDbis}. However the proof of Lemma~\ref{lmlm} (used
    in the proof of Theorem~\ref{tmDbis}) 
    is more complicated than the proof of its analog which is
    Lemma~6 of~\cite{KoSh}. This renders indispensable
    giving the whole proof of Theorem~\ref{tmDbis}.

    \begin{rem}
      {\rm For the sign pattern $D(a,b,c)$, compatible are the following pairs
        $(pos, neg)$:

        1) the ones mentioned in Theorem~\ref{tmDbis};

        2) the pair $(1,0)$;

        3) the pairs $(2j+1,2r)$, $r=1$, $2$, $\ldots$, $a+c-1$,
        $j=0$, $1$, $\ldots$, $b$.

        Realizability of the couples $(D(a,b,c),(pos,neg))$ with $(pos,neg)$
        as in 2) and 3) can be proved by analogy with the proof of parts (ii)
      and (iii) of Theorem~4 in~\cite{KoSh}.}
    \end{rem}

\section{Proof of Theorem~\protect\ref{tmbasic}\protect\label{secprtmbasic}}

\begin{proof}[Part (1)]
A) For $d\geq 6$, the set $\Pi ^*_d(\sigma _{\bullet})$ is non-empty, see
Proposition~1 in~\cite{KoSe}. Fix
a polynomial $Q^*\in \Pi ^*_d(\sigma _{\bullet})$. By Proposition~1
of~\cite{KoSe}, one can choose $Q^*$ such that 
      the moduli of its positive and negative roots (denoted by
      $\alpha _1<\alpha _2<\alpha _3<\alpha _4$ and
      $\beta _1<\beta _2<\cdots <\beta _{d-5}<\beta _{d-4}$
      respectively) satisfy the string of inequalities

      \begin{equation}\label{eqalphabeta}
        \alpha _1<\alpha _2<\beta _1<\beta _2<\cdots <\beta _{d-5}<\beta _{d-4}
        <\alpha _3<\alpha _4~.
        \end{equation}
      So the negative roots of $Q^*$ are
      $-\beta _{d-4}<-\beta _{d-5}<\cdots <-\beta _1<0$.
      Starting with $Q^*$, we construct two polynomials $Q^1$ and $Q^2$
      of the set $A(\sigma _{\bullet},(2,d-4))$ (so this set is non-empty)
      about which we show that they belong to different components of
      $R_{3,d}$. This implies the theorem.
      \vspace{1mm}
      
      B) We consider the one-parameter
      family of polynomials

      $$\tilde{Q}_t:=Q^*+tx^2(x+\beta _1)(x+\beta _2)\cdots (x+\beta _{d-4})~,
      ~~~t\geq 0~.$$
      For
      any $t\geq 0$, one has $\sigma (\tilde{Q}_t)=\sigma _{\bullet}$.
      As $t$ increases, the roots $-\beta _1$, $-\beta _2$, $\ldots$,
      $-\beta _{d-4}$ of
      $\tilde{Q}_t$ do not move. The roots $\alpha _1$ and $\alpha _3$
      move to the right while $\alpha _2$ and $\alpha _4$ move to the left. For
      some $t_0>0$, either $\alpha _1$ coalesces with $\alpha _2$ or
      $\alpha _3$ coalesces with $\alpha _4$ or both these things take place.
      Indeed, the values of $\tilde{Q}_t$ for each fixed $x\geq \alpha _1$ 
      increase
      at least as fast as
      $t\alpha _1^2\prod _{i=1}^{d-4}(\alpha _1+\beta _i)$.

      If for $t=t_0$,
      $\alpha _1$ and $\alpha _2$ coalesce and $\alpha _3$ and $\alpha _4$
      remain positive and distinct, then one can fix $t_1>t_0$ sufficiently
      close to $t_0$ for which the roots $\alpha _1$ and $\alpha _2$
      have given birth
      to a complex conjugate pair while $\alpha _3$ and $\alpha _4$ are still
      positive and distinct. We set $Q^1:=\tilde{Q}_{t_1}$. Hence the
      polynomial $Q^1$ has $d-2$ real roots 

      \begin{equation}\label{eqQ1}
        \begin{array}{ll}
          -\beta _{d-4}<-\beta _{d-5}<\cdots <-\beta _1<0<\alpha _3<\alpha _4&
        {\rm such~that}\\ \\ 
        0<\beta _{1}<\beta _{2}<\cdots <\beta _{d-4}<\alpha _3<\alpha _4&
        \end{array}
      \end{equation}
      and a complex conjugate pair. After this we set $Q^2_*:=x^dQ^1(1/x)$.
      The sequence of coefficients of $Q^1$, when read from the right, is the
      string of coefficients of $Q^2_*$. After this we set $Q^2:=Q^2_*/Q^1(0)$, 
      so $Q^2$ is monic. The sign pattern
      $\sigma _{\bullet}$ is center-symmetric, therefore
      $\sigma (Q^2)=\sigma _{\bullet}=\sigma (Q^1)$.
      The roots of the polynomial
      $Q^2$ are the reciprocals of the roots of $Q^1$.
      %By a self-evident
      %change of notation (we write $\alpha _3$, $\alpha _4$ instead of
      %$\alpha _1$, $\alpha _2$) one can say that
      The real roots of
      $Q^2$ satisfy the
      conditions

\begin{equation}\label{eqQ2}
  \begin{array}{ll}-\beta _{d-4}<-\beta _{d-5}<\cdots <-\beta _1<0<\alpha _3<
    \alpha _4&{\rm and}\\ \\ 
    0<\alpha _1<\alpha _2<\beta _{1}<\beta _{2}<\cdots <\beta _{d-4};
    \end{array}
        \end{equation}
the polynomial $Q^2$ has also a complex conjugate pair.

If for $t=t_0$, $\alpha _3$ and $\alpha _4$ coalesce while $\alpha _1$ and
$\alpha _2$ remain positive and distinct, then for some $t_1>t_0$ sufficiently
close to $t_0$ we obtain the polynomial $Q^2$ with exactly two positive
and $d-4$ negative roots which satisfy conditions (\ref{eqQ2}). After this we
set $Q^1_*=x^dQ^2(1/x)$ and $Q^1:=Q^1_*/Q^2(0)$. The real roots of $Q^1$
satisfy conditions (\ref{eqQ1}).

Finally, if for $t=t_0$, one has $\alpha _1=\alpha _2=a>0$ and
$\alpha _3=\alpha _4=b>a$, then one constructs the polynomials

$$Q^{\pm}:=\tilde{Q}_{t_0}\pm \varepsilon (x-(a+b)/2)~,~~~\varepsilon >0~.$$
For $\varepsilon$ small enough,
\vspace{1mm}

1) the coefficients of $Q^{\pm}$ are non-zero and
$\sigma (Q^{\pm})=\sigma _{\bullet}$;  
\vspace{1mm}

2) each of the polynomials $Q^{\pm}$ has
$d-4$ distinct negative roots close to $-\beta _i$;
\vspace{1mm}

3) $Q^+$ has two distinct positive roots close to~$a$
and a complex conjugate pair close to~$b$;
\vspace{1mm}

4) and vice versa for $Q^-$.
\vspace{1mm}

We set $Q^1:=Q^-$ and $Q^2:=Q^+$. 
\vspace{1mm}

C) Suppose that the two polynomials $Q^1$ and $Q^2$ belong to one and the same
component of the set $R_{3,6}$. Then it is possible to connect
them by a continuous path (homotopy) within this component:
$Q^s$, $s\in [1,2]$. Along the path the two
positive, the $d-4$ negative and the two complex conjugate roots of $Q^s$
depend continuously on~$s$ while remaining distinct throughout the homotopy.
We denote the negative roots by
$-\tilde{\beta}_j$, $j=1$, $\ldots$, $d-4$, and the two positive roots by
$\tilde{\gamma}_j$,
$j=1$, $2$, where

$$\begin{array}{lllll}
{\rm for}~~~&s=1~,&{\rm one~has}&\tilde{\beta}_j=\beta _j~,&
\tilde{\gamma}_j=\alpha _{2+j}~;\\ \\ 
      {\rm for}~~~&s=2~,&{\rm one~has}&\tilde{\beta}_j=\beta _j~,&
      \tilde{\gamma}_j=\alpha _j~.
    \end{array}$$
    Hence there exists $s=s_0\in (1,2)$ such that for $s=s_0$,
    $\tilde{\beta}_{d-4}=\tilde{\gamma}_2$. This means that the polynomial
    $Q^{s_0}$ has exactly $d-2$ real roots such that

    $$-\tilde{\beta}_{d-4}<\cdots <-\tilde{\beta}_1<0<\tilde{\gamma}_1<
    \tilde{\gamma}_2~~~,~~~\, \,
    \tilde{\beta}_{d-4}=\tilde{\gamma}_2~.$$
    Using a linear change $x\mapsto hx$, $h>0$, we achieve the condition
    $\tilde{\beta}_{d-4}=\tilde{\gamma}_2=1$.
    \vspace{1mm}

    D) Suppose that $d$ is even. The fact that $\pm 1$ are roots of $Q_d$
    implies the two conditions:

    $$a_d+a_{d-2}+a_{d-4}+\cdots +a_2+a_0=0~~~\, \, {\rm and}~~~\, \, 
    a_{d-1}+a_{d-3}+\cdots +a_3+a_1=0~.$$
    The first of them is possible only if all even coefficients are $0$,
    because in the corresponding positions the sign pattern $\sigma _{\bullet}$
    contains $(+)$-signs. However $a_d=1$. This contradiction means that the
    homotopy $Q^s$ does not exist, so $Q^1$ and $Q^2$ belong to different
    components of the set $R_{3,d}$ and the set $A(\sigma _{\bullet},(2,d-4))$
    is not connected. One can observe that this resoning is not valid for
    $d=2$ or $d=4$, because in these cases there are no negative roots at all.
    \vspace{1mm}
    
    E) Suppose that $d\geq 7$ is odd. Set $\delta :=\tilde{\beta}_{d-5}>0$ 
    %(so this
    %reasoning requires the %condition $d>5$, i.e. $d\geq 7$) 
    and
    $Q_d^{s_0}=(x+\delta )U(x)$, where $U=x^{d-1}+\sum _{j=0}^{d-2}u_jx^j$.
    The polynomial $U$ has an even number of positive roots, so $u_0>0$. The
    conditions

    $$0>a_{d-1}=\delta +u_{d-2}~~~\, \, {\rm and}~~~\, \, \delta >0$$
    imply $u_{d-2}<0$ whereas from

    $$0<a_{d-2}=\delta u_{d-2}+u_{d-3}~,~~~\delta >0~~~\, \, {\rm and}~~~\, \,
    u_{d-2}<0$$
    one deduces that $u_{d-3}>0$. In the same way one has

    $$\begin{array}{ll}
      0>a_1=\delta u_1+u_0~,~~~\delta >0~,~~~u_0>0~,~~~\, \, {\rm so}~~~
      \, \, u_1<0&
      {\rm and}\\ \\
      0<a_2=\delta u_2+u_1~,~~~\delta >0~,~~~u_1<0~,~~~\, \, {\rm so}~~~
      \, \, u_2>0~.&
      \end{array}$$
    The first three and the last three of the coefficients of the polynomial
    $U(-x)$ are positive. By Descartes' rule of signs it has not more than
    $d-5$ positive roots, and it has exactly $d-5$ positive roots only if
    it has $d-5$ sign changes. On the other hand one knows that $U(-x)$ has
    exactly $d-5$ positive roots $-\tilde{\beta}_j$, $j=1$, $2$,
    $\ldots$, $d-6$, $d-4$.
    Hence $U(x)$ has $d-5$ sign preservations, therefore $u_k>0$ for
    $2\leq k\leq d-3$.

    Thus $\sigma (U)=\sigma _{\bullet}$ (but here the sign
    pattern $\sigma _{\bullet}$ is meant to be of length $d$, not $d+1$).
    Suppose that the homotopy $Q^s$ exists. Along this homotopy the root
    $-\tilde{\beta}_{d-5}$ is a continuous negative-valued function. As division
    of $Q^s$ by $x+\delta$ gives the polynomials $U$, there exists a homotopy
    between the polynomial $U$ corresponding to $Q^1$ and the one
    corresponding to $Q^2$. We denote them by $U^1$ and $U^2$. They are of
    even degree $d-1\geq 6$, each of them has exactly two positive roots
    $\tilde{\gamma}_1<\tilde{\gamma}_2$, exactly $d-5$ negative roots and
    one complex conjugate pair. For the moduli of the real roots one has

    $$\tilde{\beta}_j<\tilde{\gamma}_1~~~\, \, {\rm for}~~~\, \, U^1~~~\, \,
    {\rm and}
    ~~~\, \, \tilde{\gamma}_2<\tilde{\beta}_j~~~\, \, {\rm for}~~~\, \, U^2~,~~~
    j=1,~2,\ldots ~,~d-6,~d-4$$
    (see (\ref{eqQ1}) and (\ref{eqQ2})). This, however, is impossible, see~D).
\end{proof}

\begin{proof}[Part (2)]
  F) For $d=4$, for each polynomial $Q\in A(\sigma _{\bullet},(2,0))$, there
  exists a unique quantity $g>0$ such that for $g'\in [0,g)$, one has
    $Q+g'\in A(\sigma _{\bullet},(2,0))$ and for $g'=g$, the polynomial
    $Q+g'$ has a multiple positive root. 
    %and for $g'>g$, it has no real roots.

    On the other hand, for each polynomial $Q\in A(\sigma _{\bullet},(2,0))$,
    there exists a unique quantity $h>0$ such that for $h'\in [0,h)$, one has
      $Q-h'\in A(\sigma _{\bullet},(2,0))$ and 
      for $h'=h$, $Q$ has either a zero root
      or a multiple positive root. %and for $h'>h$,
      %it is not a polynomial of the set $A(\sigma _{\bullet},(2,0))$. 
      The
      quantities $g$ and $h$ are  continuous functions
      of the coefficients of~$Q$.
      
    Denote by $A^*(\sigma _{\bullet})$ the set of monic polynomials whose
    coefficients have signs as defined by the sign pattern $\sigma _{\bullet}$
    and which have a multiple positive root and a complex conjugate pair.
    Hence the set $A(\sigma _{\bullet},(2,0))$ is homeomorphic to the
    direct product of the set $A^*(\sigma _{\bullet})$ and an open interval.
    Therefore if $A^*(\sigma _{\bullet})$ is connected, then such is
    $A(\sigma _{\bullet},(2,0))$ as well.

    Denote by $A^*_0(\sigma _{\bullet})$ the subset of $A^*(\sigma _{\bullet})$
    for which the multiple root of $Q$ is at~$1$. Each polynomial
    $Q\in A^*(\sigma _{\bullet})$ can be transformed into a polynomial of
    $A^*_0(\sigma _{\bullet})$ by a linear change of the variable $x$ followed
    by a multiplication with a non-zero constant. Hence
    $A^*(\sigma _{\bullet})$ is homeomorphic to
    $A^*_0(\sigma _{\bullet})\times (0,\infty)$.

    Any polynomial $Q\in A^*_0(\sigma _{\bullet})$
    is of the form

    $$(x-1)^2(x^2+Ax+B)=x^4+(A-2)x^3+(B-2A+1)x^2+(A-2B)x+B~,$$
    where $A^2-4B\leq 0$.
    The set $A^*_0(\sigma _{\bullet})$ is defined by the
    conditions

    $$A<2~,~~~\, \, B-2A+1>0~,~~~\, \, A-2B<0~~~\, \, {\rm and}~~~\, \,
    B\geq A^2/4~.$$
    %The latter inequality implies $B>0$.
    This is the set of points in the
    plane $(A,B)$ which are to the left of the vertical line $A=2$ and above or on 
    the graph of the function (of the argument $A\in (-\infty ,2)$)
    $\max (2A-1,A/2,A^2/4)$; strictly above for $A\in [0,2)$ and above or on the graph for $A<0$. This is a contractible set.
    \vspace{1mm}

    G) For $d=5$, we denote by $A^{\dagger}(\sigma _{\bullet})$
    the set of monic polynomials
    the signs of whose coefficients are defined by the sign pattern
    $\sigma _{\bullet}$ and which have a simple negative root,
    a double positive root and a complex conjugate pair. Denote by
    $A^{\dagger}_0(\sigma _{\bullet})$ its subset for which the double root is
    at~$1$. By complete analogy with part F) of the proof we show that
    connectedness of $A^{\dagger}_0(\sigma _{\bullet})$ implies the one of
    $A(\sigma _{\bullet},(2,1))$.

    Any polynomial $Q\in A^{\dagger}_0(\sigma _{\bullet})$ is of the form

    $$(x-1)^2(x+A)(x^2+Bx+C)=x^5+\sum _{j=0}^4f_jx^j~,~~~\, \,
      {\rm where}$$

      $$\begin{array}{ll}f_4=A+B-2~,&f_3=AB-2A-2B+C+1~,\\ \\ 
        f_2=-2AB+AC+A+B-2C~,&f_1=AB-2AC+C\\ \\ {\rm and}&f_0=AC~,
      \end{array}$$
    with $A>0$ and $B^2-4C<0$. For any $\rho >0$ and $r>0$, the polynomial
    $Q_{\rho ,r}:=Q+\rho (x-1)^2+rx^3(x-1)^2$ defines the sign pattern
    $\sigma _{\bullet}$ and
    belongs to the set $A^{\dagger}_0(\sigma _{\bullet})$. Indeed,
    it is non-negative for $x\geq 0$, with equality only for $x=1$;
    its second derivative at $x=1$ is positive, so $x=1$ is a double root;
    the sign pattern
    $\sigma _{\bullet}$ and Descartes' rule of sign imply that $Q_{\rho ,r}$
    has not more than one negative root, so it has exactly one such root.
    Hence one can choose
    $\rho$ and $r$ such that $f_2=f_3$. The set $A^{\dagger}_0(\sigma _{\bullet})$
    is connected if and only if its subset defined by the condition $f_2=f_3$
    is connected.
    \vspace{1mm}

    H) The condition $f_2=f_3$ allows to express $A$ as a function of $B$
    and~$C$:

    $$A=T_0/D~,~~~\, \, {\rm where}~~~\, \, T_0=3B-3C-1~~~\, \,
    {\rm and}~~~\, \, D=3B-C-3~.$$
    For the coefficients $f_i$ with $A=T_0/D$ one finds

    $$\begin{array}{ll}
      f_4=T_4/D~,&T_4=3B^2-BC-6B-C+5~,\\ \\
      f_3=f_2=T_3/D~,&T_3=-3B^2+2BC-C^2+2B+2C-1~,\\ \\
      f_1=T_1/D~,&T_1=3B^2-6BC+5C^2-B-C~.\end{array}$$
    In Fig.~\ref{fig45} and \ref{fig45bis} we represent the following sets:
    \vspace{1mm}
    
    -- $\mathcal{L}_0:T_0=0$ (in solid line) and
    $\mathcal{L}:D=0$ (in dashed line) are straight lines;
    \vspace{1mm}
    
    -- $\mathcal{E}_3:T_3=0$ (in dashed line) and $\mathcal{E}_1:T_1=0$
    (in dotted line) are
    ellipses;
    \vspace{1mm}
    
    -- $\mathcal{H}:T_4=0$ (in solid line) is a hyperbola;
    \vspace{1mm}

    -- $\mathcal{P}:C=B^2/4$ is a parabola (in dash-dotted line).

\begin{figure}[htbp]
%\includegraphics[scale=0.5]{parthetanegfirstfour.eps}
%\centerline{\hbox{\includegraphics[scale=0.7]{parthetanegfirstfour.eps}}}
%\vskip0.5cm
  \centerline{\hbox{\includegraphics[scale=0.7]{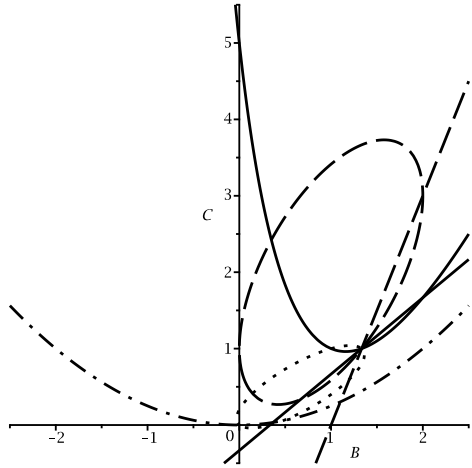}}}
  %\hbox{\includegraphics[scale=0.7]{fig45loc.pdf}}}
%\centerline{\hbox{\epsfxsize=10cm \epsfbox{k=1234.pdf}}}
  \caption{The set $A^{\dagger}_0(\sigma _{\bullet})$ subdued to the condition
    $f_2=f_3$ (global view).)}
\label{fig45}
\end{figure}

    \begin{rem}\label{remIntOut}
      {\rm As $C>0$, only the branch of $\mathcal{H}$ belonging to the upper
        half-plane is represented in Fig.~\ref{fig45} and~\ref{fig45bis}.
        The asymptotes of
        $\mathcal{H}$ are the lines $B=-1$ and $C=3B-6$. We denote by
        Int$(\mathcal{E}_i)$ and Out$(\mathcal{E}_i)$ the intersections
        with the half-plane $C>0$ of the interior and the
        exterior of the ellipse $\mathcal{E}_i$. By Int$(\mathcal{H})$ we denote
        the part of the upper half-plane which is above and by
        Out$(\mathcal{H})$ the part which is below the branch of $\mathcal{H}$
        with $C>0$. Notice that}

      $$\begin{array}{lll}
        {\rm Int}(\mathcal{E}_3):~T_3>0,~C>0,&
        {\rm Int}(\mathcal{E}_1):~T_1<0,~C>0,&
        {\rm Int}(\mathcal{H}):~T_4<0,~C>0,\\ \\
        {\rm Out}(\mathcal{E}_3):~T_3<0,~C>0,&
        {\rm Out}(\mathcal{E}_1):~T_1>0,~C>0,&
        {\rm Out}(\mathcal{H}):~T_4>0,~C>0.
      \end{array}$$
      {\rm The ellipse $\mathcal{E}_1$ intersects the $C$-axis at $(0,0)$ and
        $(0,1/5)$ while $\mathcal{E}_3$ is tangent to the $C$-axis at $(0,1)$.
        The leftmost point of the ellipse $\mathcal{E}_1$ is at
        
        $$((8-\sqrt{70})/12=-0.030\ldots ,(10-\sqrt{70})/20=0.081\ldots )~.$$The
        point $(4/3,1)$ is a common point for $\mathcal{L}$, $\mathcal{L}_0$,
        $\mathcal{H}$, $\mathcal{E}_1$ and $\mathcal{E}_3$.

        The intersecting lines $\mathcal{L}_1$ and $\mathcal{L}$
        define two pairs of opposite sectors. The ones of opening
        $>\pi /2$ are denoted by $\mathcal{S}_u:T_0<0,D<0$ (upper) and
        $\mathcal{S}_{\ell}:T_0>0,D>0$ (lower). One has $A>0$ exactly when
        the point $(B,C)$ belongs to one of these two sectors.}
    \end{rem}
    \vspace{1mm}

\begin{figure}[htbp]
%\includegraphics[scale=0.5]{parthetanegfirstfour.eps}
%\centerline{\hbox{\includegraphics[scale=0.7]{parthetanegfirstfour.eps}}}
%\vskip0.5cm
\centerline{\hbox{\includegraphics[scale=0.7]{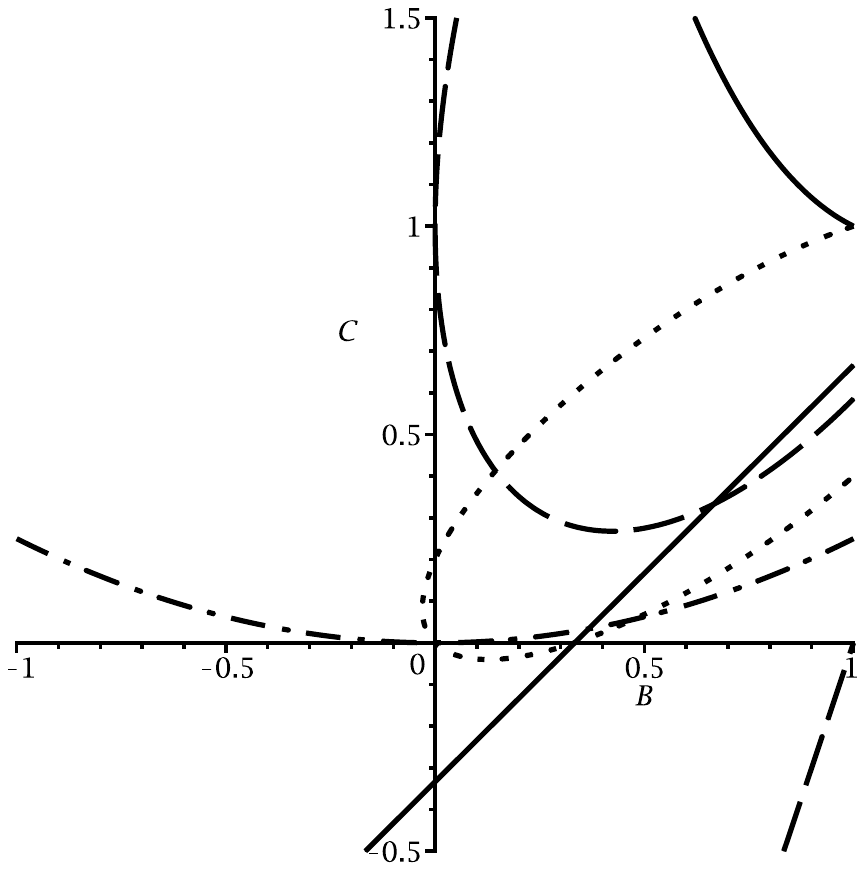}}}
  %\hbox{\includegraphics[scale=0.7]{fig45loc.pdf}}}
%\centerline{\hbox{\epsfxsize=10cm \epsfbox{k=1234.pdf}}}
    \caption{The set $A^{\dagger}_0(\sigma _{\bullet})$ subdued to the condition
    $f_2=f_3$ (local view).)}
\label{fig45bis}
\end{figure}

    I) The signs of the coefficients $f_i$ and of the quantities $A>0$ and $C>0$
    imply that one must have one of the two systems of conditions:

    $$\begin{array}{l}
      {\rm (i)}~:~(B,C)\in \mathcal{S}_{\ell}\cap {\rm Int}(\mathcal{E}_1)\cap
      {\rm Int}(\mathcal{E}_3)\cap {\rm Int}(\mathcal{H})~,~~~\, \,
      {\rm i.e.}\\ \\
      ~~~~~\, \, T_0>0~,~~~D>0~,~~~T_1<0~,~~~T_3>0~~~\, \, {\rm and}~~~\, \, T_4<0~~~\, \,
      {\rm or}\\ \\ 
    {\rm (ii)}~:~(B,C)\in \mathcal{S}_u\cap {\rm Out}(\mathcal{E}_1)\cap
      {\rm Out}(\mathcal{E}_3)\cap {\rm Out}(\mathcal{H})~,~~~\, \,
      {\rm i.e.}\\ \\ 
      ~~~~~\, \, T_0<0~,~~~D<0~,~~~T_1>0~,~~~T_3<0~~~\, \, {\rm and}~~~\, \, T_4>0~.
    \end{array}$$
    The possibility (i) is to be excluded. Indeed, one has

    $$\mathcal{E}_3\cap \mathcal{L}_0=\{ (2/3,1/3)~,~~~(4/3,1)\} ~~~\, \,
    {\rm and}~~~\, \,
    \mathcal{E}_3\cap \mathcal{L}=\{ (4/3,1)~,~~~(2,3)\} ~,$$
    see Fig.~\ref{fig45} and~\ref{fig45bis},
    so Int$(\mathcal{E}_3)$ intersects with
    $\mathcal{S}_u$, but not with $\mathcal{S}_{\ell}$.
    \vspace{1mm}

    J) We describe the set obtained in case (ii). For $B\leq -1$,
    this is the part of the upper plane which is above the parabola
    $\mathcal{P}$. For $-1<B<(8-\sqrt{70})/12$,
    this is its part between the parabola $\mathcal{P}$ from below and
    the hyperbola $\mathcal{H}$ from above, see Fig.~\ref{fig45}. For
    each $(8-\sqrt{70})/12\leq B<0$, this is the union of two intervals whose
    endpoints belong to $\mathcal{H}$ and $\mathcal{E}_1$ for the upper and to
    $\mathcal{E}_1$ and $\mathcal{P}$ for the lower interval. For
    $B\geq 0$, this is the union of two curvilinear triangles, each with
    one rectilinear side which is part of the $C$-axis. The above triangle
    has vertices at $(0,1)$, $(0,5)$ and
    $(0.34\ldots ,2.42\ldots )$. The latter point, together
    with $(4/3,1)$, is the
    intersection $\mathcal{H}\cap \mathcal{E}_3$. 

    The lower triangle has vertices at $(0,1/5)$, $(0,1)$ and
    $(0.14\ldots ,0.41\ldots )$.
    The latter point, together with $(4/3,1)$, is the intersection
    $\mathcal{E}_1\cap \mathcal{E}_3$.

    To see that there is no other point of the set defined in case (ii) with
    $B>0$, one has to observe the order on $\mathcal{P}$ of the intersection
    points of

    $$\begin{array}{l}\mathcal{P}\cap \mathcal{L}_0=\{ (0.36\ldots ,~0.03\ldots ),~
    (3.63\ldots ,~3.29\ldots )\}~~~\, \, {\rm and}\\ \\ 
    \mathcal{P}\cap \mathcal{E}_1=\{ (0,~0),~
    (0.47\ldots ,~0.22\ldots )\} ~.\end{array}$$
    The connectedness of the set obtained in case (ii) follows from its description. 
  \end{proof}

    \section{Proof of Theorem~\protect\ref{tm21}\protect\label{secprtm21}}

    Part (1).
    %We give the proof only for the pair $(2,1)$, for $(1,2)$
    %it is deduced
  %via the mapping $x\mapsto -x$, see Definition~\ref{defiZaction}.
  The last component of $\sigma _{\triangle}$
  is a~$+$. Suppose that there is a minus sign in $\sigma _{\triangle}$
  corresponding to $x^{2m}$, $1\leq m\leq [d/2]$. The polynomial $-x^{2m}+1$
  has exactly two real roots, namely $\pm 1$, and they are simple. For
  $\varepsilon >0$ small enough, the polynomial $P_0:=\varepsilon x^d-x^{2m}+1$
  has exactly three real roots two of which are close to $\pm 1$ and the third
  is $>1$. (One can notice that by Descartes' rule of signs it has not more
  than two positive and not more than one negative root.)

  Fix a degree $d$ polynomial $P_1$ with $\sigma (P_1)=\sigma _{\triangle}$.
  Then for $0<\eta\ll \varepsilon$, the
  polynomial $P_0+\eta P_1$ has signs of the coefficients as defined by
  $\sigma _{\triangle}$ and has exactly one negative and two positive simple
  roots and $(d-3)/2$ complex conjugate pairs counted with multiplicity.
  Thus $P_0+\eta P_1$ realizes the couple
  $(\sigma _{\triangle},(2,1))$.

  Suppose now that there are $(+)$-signs in $\sigma _{\triangle}$ corresponding
  to all monomials of even degrees. Then there is a monomial $x^{2m+1}$,
  $1\leq 2m+1<d$, whose sign is negative. The polynomial $P_2:=x^d-x^{2m+1}$
  has simple roots at $\pm 1$ and a $(2m+1)$-fold root at $0$. For
  $\varepsilon >0$ small enough, the polynomial $P_2+\varepsilon$ has exactly three
  real roots (two positive and one negative) all of which are simple. Then with $P_1$ and $\eta$ as above, the
  polynomial $P_2+\varepsilon +\eta P_1$ realizes the couple
  $(\sigma _{\triangle},(2,1))$. 
  \vspace{1mm}
  
Part (2). We construct a polynomial of the form $V:=x^d-Ax^{2m}-Bx^{2n-1}+C$,
  $A>0$, $B>0$, $C>0$, such that $V(1)=V'(1)=V(-1)=0$:

  \begin{equation}\label{eq3eq}\begin{array}{lll}
    1-A-B+C=0~,&-1-A+B+C=0~,&d-2mA-(2n-1)B=0\\ \\
    {\rm hence}&A=C=(d-2n+1)/2m~,&B=1~.\end{array}\end{equation}
  By Descartes' rule of signs, $V$ has no other real roots. After this one decreases $C$: $C\mapsto C-t$, $t\geq0$. For $t=0$,
  the root $-1$ moves
  with a finite speed to the right while the double root at $1$ splits into
  two real roots moving for $t=0$ with infinite speeds to the
  left and right respectively. Hence for $t>0$ close to $0$, one has
  $\alpha _1<\beta <\alpha _2$. The linear system (\ref{eq3eq}) with
  unknown variables $A$, $B$ and $C$ has non-zero
  determinant. Hence for $\varepsilon >0$ small enough, one can obtain
  polynomials $V$ satisfying the conditions

  $$V(1)=V'(1)=V(-1\pm \varepsilon )=0~~~\, \, {\rm (resp.}~~~\, \,
  V(1)=V'(1)\pm \varepsilon =V(-1)=0{\rm )}$$
  which after decreasing $C$ yield
  polynomials satisfying the inequalities $\beta <\alpha _1<\alpha _2$
  or $\alpha _1<\alpha _2<\beta$ (resp. the conditions
  $\beta =\alpha _1<\alpha _2$
  or $\alpha _1<\alpha _2=\beta$). It remains to construct the polynomial
  $V+\eta P_1$, where $0<\eta \ll \varepsilon$ and
  $\sigma (P_1)=\sigma _{\triangle}$. 
  \vspace{1mm}
  
  Part (3). There exists a monomial $x^{2m}$ with negative coefficient.
  Then for $\varepsilon >0$ small enough, the polynomial
  $W:=x^{2m-1}(x-1)(x-2)+\varepsilon$ has exactly one negative and two positive
  roots whose
  moduli satisfy the condition $\beta <\alpha _1<\alpha _2$. Its four non-zero
  coefficients have the signs as defined by $\sigma _{\triangle}$. After this one
  constructs the polynomial $W+\eta P_1$ with $\eta$ and $P_1$ as above.

  The inequality $\beta \geq \alpha _1$ is impossible. Indeed, represent a
  polynomial $W$ realizing the couple $\mathcal{C}$ in the form
  $W=W_o+W_e$, where $W_o$ is the odd and $W_e$ is the even part of $W$.
  Then for $x\in (-\beta ,0)$, one has $W_e(x)=W_e(-x)$ and $W_o(x)<W_o(-x)$.
  As $W(x)>0$ for $x\in (-\beta ,0)$, one cannot have $W(\alpha _1)=0$.
  This is a contradiction.
  \vspace{1mm}
  
  Part (4). Changing the polynomial $Y(x)$ with $\sigma (Y)=\sigma _{\triangle}$
  which realizes the couple
  $\mathcal{C}$
  to $Y_1:=x^dY(1/x)/Y(0)$ (we set $\sigma _{\triangle}^R:=\sigma (Y_1)$)
  one obtains a polynomial realizing the couple
  $(\sigma _{\triangle}^R,(2,1))$, where all odd monomials have positive signs,
  see Definition~\ref{defiZaction}. The roots of $Y_1$ are the reciprocals
  of the roots of $Y$, so one deduces
  part (4) from part~(3).

  \section{Proof of Theorem~\protect\ref{tmD}\protect\label{secprtmD}}

  %We prove the theorem only for the pair $(3,0)$, for $(0,3)$ the
  %proof is deduced
  %via the mapping $x\mapsto -x$, see definition~\ref{defiZaction}.

  The last sign of $\sigma _{\diamond}$ is a~$-$. Suppose that there are two
  monomials $x^{2m}$ and $x^{2p}$, $m>p>0$, whose signs defined by
  $\sigma _{\diamond}$ are $-$ and $+$ respectively. Consider the polynomial
  $P_3:=-x^{2m}+Ax^{2p}-B$, $A>0$, $B>0$. By Descartes' rule of signs it has at
  most two positive and at most two negative roots. We define $A$ and $B$
  such that $P_3$ has double roots at $1$
  and~$(-1)$:

  $$\begin{array}{lll}-1+A-B=0~,&-2m+2pA=0&{\rm hence}\\ \\
    A=m/p>0~,&B=(m-p)/p>0~.&\end{array}$$
  Then for $\varepsilon >0$ small enough, the polynomial $P_3+\varepsilon x^d$
  has exactly three real roots, all simple and positive. Suppose that $P_4$
  is a degree $d$
  polynomial such that $\sigma (P_4)=\sigma _{\diamond}$. Then for
  $0<\eta \ll \varepsilon$, the polynomial $P_3+\varepsilon x^d+\eta P_4$
  has sign pattern $\sigma _{\diamond}$ and has exactly three real roots, all
  simple and positive.

  Suppose that there are no monomials $x^{2m}$ and $x^{2p}$ as above.
  Then the signs of the first $a$ even monomials are positive and the ones of
  the last $(d+1-2a)/2$ of them are negative, $0\leq a\leq (d-1)/2$.
  Suppose that there are monomials $x^{2\nu}$, $x^{2\mu -1}$ and $x^{2\theta}$,
  $2\nu >2\mu -1>2\theta$, whose signs defined by $\sigma _{\diamond}$ are
  $-$, $+$ and $-$ respectively. By Descartes' rule of signs a polynomial
  of the form $P_5:=-x^{2\nu}+Cx^{2\mu -1}-Dx^{2\theta}$, $C>0$, $D>0$,
  has at most two
  positive roots and no negative roots; clearly it has a
  $(2\theta )$-fold root at $0$. One can choose $C$ and $D$ such that
  the positive roots are at $1$ and~$2$:

  $$\begin{array}{lll}
    -1+C-D=0~,&-2^{2\nu}+2^{2\mu -1}C-2^{2\theta}D=0&
  {\rm hence}\\ \\ D=(2^{2\nu}-2^{2\mu -1})/(2^{2\mu -1}-2^{2\theta})>0~,&
  C=D+1>0~.&\end{array}$$
  For $\varepsilon >0$ small enough, the polynomial $P_5+\varepsilon x^d$
  has three positive simple roots and no other real roots, and the polynomial
  $P_6:=P_5+\varepsilon x^d+\eta P_4$ with $\eta$ and $P_4$ as above has three
  positive simple roots, no other real roots and
  $\sigma (P_6)=\sigma _{\diamond}$.

  So now we suppose that there are no monomials $x^{2m}$ and $x^{2p}$, and no
  monomials  $x^{2\nu}$, $x^{2\mu -1}$ and $x^{2\theta}$ as above. Suppose that
  there are monomials $x^{2u-1}$ and $x^{2v-1}$, $d>2u-1>2v-1>0$, such that their
  signs are $-$ and $+$ respectively. One can construct a polynomial
  $P_7:=x^d-Ex^{2u-1}+Fx^{2v-1}$, $E>0$, $F>0$, having double roots at $\pm 1$, a
  $(2v-1)$-fold root at $0$ and no other real roots:

  $$\begin{array}{lll}
    1-E+F=0~,&d-(2u-1)E+(2v-1)F=0&{\rm hence}\\ \\ 
  F=(d-2u+1)/2(u-v)>0~,&E=F+1>0~.&\end{array}$$
  The absence of other real roots is guaranteed by Descartes' rule of signs.
  Hence for $0<\eta \ll \varepsilon \ll 1$, the polynomial
  $P_7-\varepsilon +\eta P_4$ has sign pattern $\sigma _{\diamond}$,
  three
  simple positive roots and no other real roots (recall that $P_4(0)<0$).

  Suppose that there are no couples or triples of monomials  $x^{2m}$, $x^{2p}$
  or $x^{2\nu}$, $x^{2\mu -1}$, $x^{2\theta}$ or $x^{2u-1}$, $x^{2v-1}$.
  Then the signs of the first $h_o\geq 1$ odd monomials (including $x^d$)
  are positive and the signs of the remaining $(d+1-2h_o)/2$
  odd monomials are negative. The signs of the first $h_e\geq 0$ even monomials
  are positive and the signs of the other $(d+1-2h_e)/2$ ones are negative.
  The absence of triples
  $x^{2\nu}$, $x^{2\mu -1}$, $x^{2\theta}$ implies $h_o\leq h_e+1$. The cases
  $h_o=h_e+1$ and $h_o=h_e$ are impossible, because there is only one sign
  change in the sign pattern. Therefore $1\leq h_o\leq h_e-1$.
  This means that the
  sign pattern is $D(a,b,c)$ with $a=h_o$, $b=h_e-h_o$ and $c=(d+1-2a-2b)/2$.

\section{Proof of Theorem~\protect\ref{tmDbis}\protect\label{secprtmDbis}}

Suppose that a polynomial $P:=\sum _{j=0}^da_jx^j$ realizes the couple
$(D(a,b,c),(3,0))$. 
 % has the sign pattern $\sigma _k$ and realizes the pair
 % $(2s+1,0)$, $1\leq s\leq k$.
  Denote by
  $$P_o:=\sum _{\nu =0}^{(d-1)/2}a_{2\nu +1}x^{2\nu +1}~~~{\rm and}~~~
  P_e:=\sum _{\nu =0}^{(d-1)/2}a_{2\nu}x^{2\nu}$$
  its odd and even parts respectively. In each of
  the sequences $\{ a_{2\nu +1}\} _{\nu =0}^{(d-1)/2}$ and
  $\{ a_{2\nu}\} _{\nu =0}^{(d-1)/2}$ there is exactly one sign change.
  Descartes' rule of signs implies that the polynomial $P_o$ has exactly three
  real roots, namely $-x_o$, $0$ and $x_o$, $x_o>0$, while
  the polynomial $P_e$ has exactly two real roots $\pm x_e$, $x_e>0$;
  all these five roots are simple. 
  %Therefore by
  %Descartes' rule of signs
  %each of the polynomials $P_e$ and $P_o$ has exactly one real positive root
  %(denoted by $x_e$ and $x_o$ respectively) which is simple.

  \begin{rems}\label{rems:xoxe}
    {\rm (1) The polynomial $P_e$ is positive and increasing on
      $(x_e,\infty )$ and negative on $[0,x_e)$. The polynomial $P_o$ is
        positive and increasing on $(x_o,\infty )$
        and negative on $(0,x_o)$. 

  (2) One has $x_o\neq x_e$, otherwise $P(-x_o)=0$, i.e. $P$ has a negative
        root which is a contradiction.

        (3) One can assume that all positive roots of $P$ are distinct. Indeed,
        if this is not the case, then one can perturb $P$ to make all its
        positive roots
        distinct without changing the signs of its coefficients as follows.
        If $P$ has an $\ell$-fold root $\lambda >0$ ($\ell >1$), i.e.
        $P=(x-\lambda )^{\ell}P^0$, $P^0(\lambda )\neq 0$,
        then for $\varepsilon >0$ small enough, the
        polynomial $(x-\lambda )^{\ell -1}(x-\lambda -\varepsilon )P^0$ has
      the same sign pattern and its $\ell$-fold root has split into an
      $(\ell -1)$-fold and a simple real roots. It remains to iterate this
      construction sufficiently many times.}
\end{rems}

  \begin{nota}
    {\rm We denote by $0<\xi _1<\xi _2<\xi _3$ the 
      smallest three of the positive roots of $P$ and
      by $\zeta$ a positive number different from $x_o$ and~$x_e$.}
    \end{nota}

  It is clear that $P(\zeta )>0$ for $\zeta \in (\xi _1,\xi _2)$ 
and $P(\zeta )<0$ for $\zeta \in (\xi _2,\xi _3)$. 
For $\zeta \in (\xi _1,\xi _2)$, it is 
impossible to have $P_e(\zeta )\leq 0$ and $P_o(\zeta )\leq 0$
(with at most one equality, see part (2) of Remarks~\ref{rems:xoxe}).
It is also 
impossible to have $P_e(\zeta )\geq 0$ and $P_o(\zeta )\geq 0$. Indeed, 
this would imply that $x_e\leq \zeta <\xi _2$ and $x_o\leq \zeta <\xi _2$
which means that for $x\in (\xi _2,\xi _3)$, one has   
$P_e(x)\geq 0$ and $P_o(x)\geq 0$, i.e. $P(x)>0$. This is a contradiction.

Two possible situations are left:
\vspace{1mm}

{\rm a)} $P_e(\zeta )>0$, $P_o(\zeta )<0$; 
\vspace{1mm}

{\rm b)} $P_e(\zeta )<0$, $P_o(\zeta )>0$
\vspace{1mm}

\noindent   
(we skip the cases of equalities, because they were already taken into account). 

Situation a) cannot take place, because this would mean that  
$$P(-\zeta )=P_e(\zeta )-P_o(\zeta )>0~,$$
and since $P(0)<0$ and $P(x)\rightarrow -\infty$ for $x\rightarrow -\infty$,
in each of the intervals $(-\infty ,-\zeta )$ and $(-\zeta ,0)$ the 
polynomial $P$ would have at least one root -- a contradiction.

So suppose that we are in situation {\rm b)}, so  
$x_o<\zeta <x_e$. Without loss of generality one 
can assume that $\xi _1 =1$; this can be achieved by a rescaling  
$x\mapsto \xi _1 x$.  
Hence $P_o(1)=\beta >0$ and $P_e(1)=-\beta$. Considering the polynomial
$P/\beta$ instead of $P,$  
one can assume that $\beta =1$. One deduces from Lemma~\ref{lmlm} which
follows 
that there are no real roots of $P$ larger than $1$ (one can use the Taylor
series of $P$ at $1$); this 
contradiction completes the proof.

\begin{lm}\label{lmlm}
Under the above assumptions,   $P^{(m)}(1)>0$, for any $m=1, 2, \ldots, d$.
\end{lm}

\begin{proof}[Proof of Lemma~\ref{lmlm}.]
  In the proof we allow zero values of the coefficients as well. This
  is because we need to deal with compact sets on which minimization
  arguments are to be applied.

  Suppose that the sum $\delta _1:=a_1+a_3+\cdots +a_{2b+2c-1}$ is fixed
  (recall that these are all the negative coefficients of $P_o$). Then for any
  $m=1$, $2$, $\ldots$, $d,$ it is true that
  $P_o^{(m)}(1)$ is minimal for
  $$a_{2b+2c-1}=\delta _1~,~~~\,  
  a_1=a_3=\cdots =a_{2b+2c-3}=0~.$$
  Indeed, when computing the values of the
  derivatives at $x=1$, monomials of larger degree in $x$ 
are multiplied by larger factors (equal to these degrees). We apply here 
$(d-3)/2$ times the fact that for $A+B$ fixed,
the inequalities $A\geq 0$, $B\geq 0$ and
$\lambda >\mu >0$ imply that the sum $\lambda A+\mu B$ is maximal when $B=0$.

Similarly, if the sum $\delta _2:=a_{2b+2c+1}+a_{2b+2c+3}+\cdots +a_d$
of all positive
coefficients of $P_o$ is fixed, then
$P_o^{(m)}(1)$ is minimal for $a_{2b+2c+1}=\delta _2$,
$a_{2b+2c+3}=\cdots =a_d=0$.

For the polynomial $P_e$ we obtain in the same way that if the sums
$$\delta _3:=a_0+a_2+\cdots +a_{2c-2}~~~\, {\rm and}~~~\,
\delta _4:=a_{2c}+\cdots +a_{d-1}$$
are fixed, then $P_e^{(m)}(1)$ is minimal for $a_{2c-2}=\delta _3$,
$a_0=a_2=\cdots =a_{2c-4}=0$, $a_{2c}=\delta _4$, $a_{2c+2}=\cdots =a_{d-1}=0$.
Thus the polynomials $P_o$ and $P_e$ are of the form

$$P_o=Ex^{2b+2c+1}-Fx^{2b+2c-1}~~~,~~~
P_e=Gx^{2c}-Hx^{2c-2}~,$$
with $E:=a_{2b+2c+1}\geq 0$, $-F:=a_{2b+2c-1}\leq 0$, $G:=a_{2c}\geq 0$ and
$-H:=a_{2c-2}\leq 0$. Recall that

$$P(1)=0~,~~~P_o(1)=1~~~{\rm and}~~~P_e(1)=-1~,~~~{\rm i.~e.}~~~
E-F=1~~~{\rm and}~~~G-H=-1~.$$
The values of the derivatives at $x=1$ are of the form

$$P^{(m)}(1)=u_mE-v_mF+w_mG-t_mH~,~~~\, \, u_m>v_m>w_m>t_m~,
$$
with $u_m,~v_m,w_m,~t_m\in \mathbb{N}$. Hence

$$\begin{array}{lll}P^{(m)}(1)&=&(u_m-v_m)E+v_m(E-F)+(w_m-t_m)G+t_m(G-H)\\ \\
  &=&(u_m-v_m)E+(w_m-t_m)G+(v_m-t_m)>0~.\end{array}$$ 
\end{proof}

\end{document}